\documentclass[12pt, leqno]{amsart}
\usepackage{graphicx,amssymb}
\usepackage{epstopdf}
\usepackage[usenames,dvipsnames]{color}

\def\beq{\begin{equation}}
\def\eeq{\end{equation}}

\def\bp{\mathbf{p}}
\def\bq{\mathbf{q}}
\def\cO{\mathcal{O}}
\def\cM{\mathcal{M}}
\def\bv{\mathbf{v}}

\newtheorem{lemma}{Lemma}[section]

\newtheorem{proposition}[lemma]{Proposition}

\numberwithin{equation}{section}

\begin{document}

\title[Stable regimes for hard disks]{Stable regimes for hard disks in a channel with twisting walls}

\author{N. Chernov, A. Korepanov, N. Sim\'{a}nyi}
\address{Department of Mathematics,
University of Alabama at Birmingham, Birmingham, AL 35294}

\begin{abstract}
We study a gas of $N$ hard disks in a box with semi-periodic boundary
conditions. The unperturbed gas is hyperbolic and ergodic (these facts are
proved for $N=2$ and expected to be true for all $N\geq 2$). We study
various perturbations by twisting the outgoing velocity at collisions with
the walls. We show that the dynamics tends to collapse to various stable
regimes, however we define the perturbations and however small they are.
\end{abstract}

\date{\today}

\maketitle


\textbf{Gas of hard balls is a classical model of statistical
mechanics. The chaotic motion of its molecules can explain many
phenomena observed in real gases and fluids. We study a 2D gas of hard
disks placed in a channel whose walls act as ``Maxwell demon'' that
makes the angle of reflection different from the angle of incidence
(twisting the outgoing velocity vectors). We present quite unexpected
patterns developed by the gas dynamics under various twisting rules.}

\section{Introduction}

The gas of hard balls is a classical model of statistical mechanics.
Hard ball collisions produce strong scattering effect which makes the
system behave chaotically and quickly relax to equilibrium (at least
locally). In mathematical terms, the gas of hard balls in finite
container is widely regarded as a hyperbolic dynamical system (i.e.,
its Lyapunov exponents should not be zero) and its natural invariant
measure (Liouville measure) is expected to be ergodic and mixing.

More precisely, the celebrated Boltzmann-Sinai ergodic hypothesis
states that the gas of $N$ hard balls in a container with periodic
boundary conditions (a torus) is hyperbolic, ergodic, mixing (and
Bernoulli). Attempts to prove this conjecture have long history
\cite{Si63,Si70,Sz96}, and at present it is proven under various
conditions \cite{KSS91,Sim02,Sim03,Sim04,Sim06}, but not yet in its
full generality.

We note that the gas of $N$ hard balls on a $d$-dimensional torus has
$d+1$ integrals of motion: its kinetic energy and its total momentum (a
$d$-vector) are preserved, hence the position of the center of mass can
be fixed. This leads to a great reduction of the phase space
eliminating a total of $2d+1$ dimensions.

If the container has rigid walls, the total momentum is no longer
preserved, thus the phase space has higher dimensionality and is more
complicated; such gases are even harder to study. Hyperbolicity and
ergodicity have been proven only for $N=2$ balls in a rectangular box
(when the disks are no too large) \cite{Sy99} and for $N\geq 2$ disks
in a very special 2D container with curved walls where each disk is
confined to its own cell \cite{BLPS}.

We are interested in a gas of $N$ hard disks in a 2D rectangular container
with \emph{partial periodicity}, i.e., where two opposite walls are rigid
but in the other direction boundary conditions are periodic (such a
container can be regarded as a \emph{cylinder}, rather than a
\emph{torus}). The hyperbolicity and ergodicity for such a gas are proved
only for $N=2$, see \cite{Sy99}, but these properties are undoubtedly
valid for all $N \geq 2$.

The dynamics of this gas becomes more intriguing if
small driving forces are added at the rigid walls,
i.e., when collisions of the disks with the walls are
modified by stochastic or deterministic perturbations.
We consider deterministic perturbations where the
angle of reflection is no longer equal to the angle of
incidence, i.e., the velocity vector of the colliding
disk is ``twisted'' a little right after the
collision. Such twisting collision rules may appear
physically unrealistic (they belong to the
``Maxwell-demon'' type of external forces), but they
produce very realistic and interesting nonequilibrium
phenomena such as shear flow and entropy production
\cite{CL97}.

It is commonly expected (and observed empirically \cite{CL97}) that the
gas of hard balls under small perturbations remains chaotic and has a
unique nonequilibrium stationary state, perhaps in the form of a
Sinai-Ruelle-Bowen (SRB) measure, i.e., an ergodic measure with smooth
conditional densities on unstable manifolds. These facts have been
actually proven for $N=1$ particle in the the 2D periodic Lorentz gas
under two types of small perturbations: external fields
\cite{CELS,Ch01,CD09} and twisting collision rules \cite{Zh}. Similar
results were obtained for various classes of bounded billiard tables
with twisting walls \cite{AMS,MPS}.

We note that the unperturbed Lorentz gas with $N=1$ particle is a
uniformly hyperbolic dynamical system. Under small perturbations, uniform
hyperbolicity usually survives and makes the construction of an SRB
measure possible.

The gas of $N\geq 2$ hard balls, on the other hand, is \emph{never}
uniformly hyperbolic. Non-uniform hyperbolicity can be easily destroyed
even by arbitrary small perturbations making it hard to control the
perturbed dynamics. For this reason there are no theoretical proofs of
hyperbolicity or the existence of SRB measures under any perturbations,
to our knowledge. Moreover, one may expect that the perturbed gas is
not always fully hyperbolic or ergodic, i.e., that there may be
elliptic islands or multiple ergodic components, etc.

The modest purpose of this paper is to point out that things are
actually much worse, even for the gas of $N=2$ disks. We show that if
the disks are not too big (if their diameters are less than 0.5), then
an arbitrarily small twist added to collisions at the walls tends to
destroy the chaotic behavior of the gas causing a complete collapse of
the dynamics in a way that almost every phase trajectory converges to
some trivial stable regimes. The collapse seems to happen for every
type of small twists, though the limit stable regimes may be very
different and sometimes quite bizarre.

Most of the stable regimes we observed for $N=2$ also appear for any
number $N>2$ of disks (provided their diameter is small enough). On the
other hand, when the disks are not too small, then our stable regimes
seem to disappear. We believe that if the diameters of the disks are
greater than $1/N$, then the gas is fully chaotic and SRB measures
exist, in agreement with numerical evidence \cite{CL97}.

This work is motivated by discussions with J. Lebowitz and Ya. Sinai.
Our original purpose was to prove hyperbolicity and construct SRB
measures for at least some types of twisting collisions. However, we
discovered that at sufficiently low densities all twists cause a total
collapse (often to our surprise). Thus in the end we decided to report
our findings in this paper.

\section{Model}  \label{Model}

Our model is a system of $N$ hard disks in a unit square
$$
  D = \{0\leq x\leq 1,\ \ 0\leq y\leq 1\}
$$
that has rigid (reflecting) walls at $y=0$ and $y=1$ and periodic
boundary conditions at $x=0$ and $x=1$. The disks are identical (have
the same mass and radius), and they collide elastically with each
other.

We can represent our model as an infinite chain of copies of the square
$D$ placed in the infinite strip (channel) $I = \{0\leq y\leq 1\}$
where hard disks appear periodically; see Fig.~\ref{FigChannel}.

\begin{figure}[htb]
    \centering
   \includegraphics[width=5in]{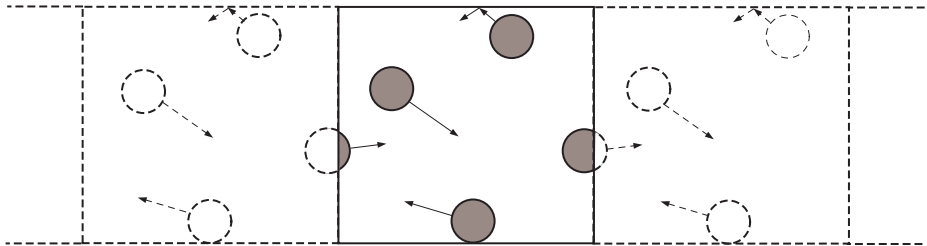}
    \caption{A periodic system of disks moving in a channel.}
    \label{FigChannel}
\end{figure}

We denote by $\bq_i = (x_i,y_i)$ the position of the $i$th disk and by
$\bp_i = (u_i,v_i)$ its velocity vector, $i=1,\ldots,N$. Now suppose a
disk collides with a wall (either $y=0$ or $y=1$). By the classical law
its velocity $\bp=\bp_{\rm before}=(u,v)$ changes to $\bp_{\rm
after}=(u,-v)$, and its kinetic energy $\tfrac 12 \|\bp\|^2$ is
preserved, i.e.,
\beq  \label{pp}
   \|\bp_{\rm before}\| = \|\bp_{\rm after}\|.
\eeq
The total momentum $\bp_{\rm total} = \bp_1+\cdots+\bp_N$ of the system
is \emph{not} preserved, but its $x$-component \emph{is} preserved,
i.e., the sum $u_1+\cdots+u_N$ remains constant in time.

The entire (macrocanonical) phase space is $4N$-dimensional, and the
energy surface is $(4N-1)$-dimensional. The conservation of
$u_1+\cdots+u_N$ allows us to set it to zero, after which the
$x$-coordinate of the center of mass will be constant and can be fixed,
too. Thus the reduced phase space is $(4N-3)$-dimensional. The resulting
system is expected to be hyperbolic (with $2N-2$ positive Lyapunov
exponents and the same number of negative ones), ergodic, and mixing.
These facts have been proved for $N=2$ in \cite{Sy99}.

We now modify the law of collisions with the walls. We will preserve
the kinetic energy of the disks, so that \eqref{pp} still holds true.
Thus it is enough to specify the angle of reflection $\psi$, as a
function of the angle of incidence, $\varphi$:
\beq \label{f}
    \psi = f(\varphi).
\eeq
We measure the angles $\psi$ and $\varphi$ as shown in
Figure~\ref{FigReflect}, so that the range of our angles is the
interval $[0,\pi]$. The function $f$ does not depend on the point of
collision, but it may be different for the bottom wall at $y=0$ and the
top wall at $y=1$. We will denote those two functions by $f_0$ and
$f_1$, respectively.

\begin{figure}[htb]
    \centering
   \includegraphics[width=3in]{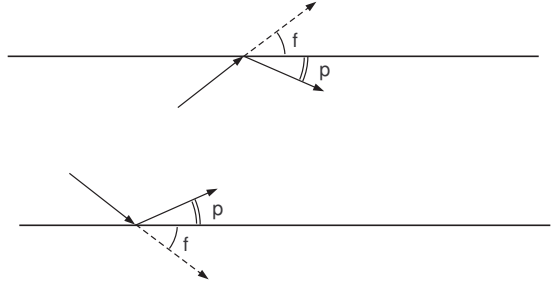}
    \caption{Angle of incidence $\varphi$ and angle of reflection $\psi$.}
    \label{FigReflect}
\end{figure}

For classical collisions $\psi=\varphi$, hence $f$ is the identity
function, $f(\varphi)=\varphi$. We will consider its small
perturbations, i.e., our functions $f$ will satisfy $f(\varphi) =
\varphi + \,{\rm small}$. For convenience we also assume that
\begin{itemize} \item[(a)] $f(\varphi)$ is continuous and strictly monotonically
increasing;
\item[(b)] $f(0)=0$ and $f(\pi)=\pi$,
\end{itemize}
i.e., $f$ is an orientation preserving homeomorphism of the interval
$[0,\pi]$. This ensures that our dynamics will be invertible, i.e.,
every phase point has a unique past.

As before, the energy surface in the phase space is
$(4N-1)$-dimensional. Since our collision rules are translation
invariant (independent of the $x$ coordinate of the collision point),
the energy surface is naturally foliated by invariant hypersurfaces on
which the dynamics are identical. Thus we can factor this foliation out
by replacing the $x$-coordinates of the disks by their \emph{relative}
$x$-coordinates $x_1-x_n$, $\ldots$, $x_{n-1}-x_n$, which eliminates
one more variable and makes the (factored) phase space
$(4N-2)$-dimensional. Most of the time we will deal with $N=2$ disks,
in which case the phase space will be 6-dimensional.

Many collision rules \eqref{f} quickly lead to trivial degenerate
regimes. For example, if $f_k(\varphi)<\varphi$ for $k=0,1$ and all
$0<\varphi<\pi$, then $u_{\rm after} = \cos\psi > \cos\varphi = u_{\rm
before}$, hence the $x$-component of the total momentum, $u_{\rm total}
= u_1+\cdots+u_N$ will grow at every collision with the walls. It is
then clear that all the particles will eventually move almost
horizontally to the right (will be ``blown away by wind'').

To prevent the wind from blowing, we will restrict our
study to collision rules where one wall
counterbalances the effect of the other:
\beq \label{ffopp}
   f_1(\varphi) = \pi - f_0(\pi-\varphi),
\eeq
i.e., the function $f_1$ at the top wall $y=1$ acts exactly opposite to
the function $f_0$ at the bottom wall $y=0$. Under this condition there
is a natural symmetry in the channel, hence drift in either direction
(left or right) cannot be dominant in the whole phase space.

\section{One particle case}  \label{SecOPC}

To clarify the effect of our twisting collisions, we
begin with the simplest case of one hard disk, i.e.,
we set $N=1$. Then the radius of the disk is
irrelevant, and we can just make it a point particle
bouncing between the two walls.

Let $\varphi_0$ be the angle of incidence at the initial collision at,
say, the bottom wall $y=0$. Then the reflection angle $\psi =
f_0(\varphi_0)$ becomes the incidence angle $\varphi_1$ at the next
collision at the top wall, i.e., $\varphi_1 = f_0(\varphi_0)$.
Similarly, the incidence angle at the following collision at the bottom
wall will be
$$
     \varphi_2 = f_1(\varphi_1) = f_1(f_0(\varphi_0)),
$$
and then the process will repeat periodically. By induction,
$$
  \varphi_{2n} = g^n(\varphi_0)
  \qquad\text{for all}\quad n\geq 1,
$$
where $g = f_1\circ f_0$, thus the evolution of incidence angles is
completely described by the iterations of the function $g$. For
functions with opposite orientations, i.e., those obeying
\eqref{ffopp}, we have
\beq \label{gfixed}
   g(\varphi)=\pi - f_0\bigl(\pi-f_0(\varphi)\bigr).
\eeq
Note that $g$, just like $f_0$, is an
orientation-preserving homeomorphism of the interval
$[0,\pi]$. But not any homeomorphism of $[0,\pi]$ can
be defined by \eqref{gfixed}, so we describe the class
of functions $g$ satisfying \eqref{gfixed}. To this
end we introduce an involution $j \colon [0,\pi]\to
[0,\pi]$ defined by $j(\varphi) = \pi-\varphi$ and
note that
$$
    g = (j\circ f_0)\circ (j\circ f_0)
$$
Thus, $g=h^2$, where $h=j\circ f_0$ is an orientation \emph{reversing}
homeomorphism of the interval $[0,\pi]$.

The map $h\colon [0,\pi]\to [0,\pi]$ obviously has a unique fixed point
$\varphi_0 \in (0,\pi)$, which automatically is a fixed point for $g$.
For any other point $\varphi \neq \varphi_0$ there are exactly three
possibilities:
\begin{itemize} \item[(a)] $h^2(\varphi)=\varphi$, then $\varphi$ is
a 2-periodic point for $h$ and a fixed point for $g$;
\item[(b)] $|h^2(\varphi)-\varphi_0|<|\varphi-\varphi_0|$, then the
    images of $\varphi$ under $g$ will move toward $\varphi_0$;
\item[(c)] $|h^2(\varphi)-\varphi_0|>|\varphi-\varphi_0|$, then the
    images of $\varphi$ under $g$ will move away from $\varphi_0$;
\end{itemize}
The non-fixed points of $g$ make an open set, which is a union of
disjoint intervals, we denote them by $\{I_m\}$. The endpoints of each
interval $I_m$ are fixed points for $g$, i.e., 2-periodic points for
$h$ (unless one of them is $\varphi_0$, of course). In each interval
$I_m\subset [0,\pi]$ all the points move under $g$ in one direction --
either toward $\varphi_0$ or away from $\varphi_0$.

For each interval $I_m$ its image $I_{m'}=h(I_m)$ is another interval
whose points move in the same direction (either toward $\varphi_0$ or
away from $\varphi_0$) as the points of $I_m$. Note that
$I_m=h(I_{m'})$ as well. We call the intervals $I_m$ and $I_{m'}$
\emph{dual} to each other. Thus the intervals $\{I_m\}$ come in
\emph{dual pairs}. Dual intervals lie on the opposite sides of
$\varphi_0$, and all the points in both intervals move either toward
$\varphi_0$ or away from $\varphi_0$. If $I_m,I_{m'}$ and $I_n,I_{n'}$
are two pairs of dual intervals, then one pair lies inside the other;
see Figure~\ref{FigInts}.

\begin{figure}[htb]
    \centering
   \includegraphics[width=4in]{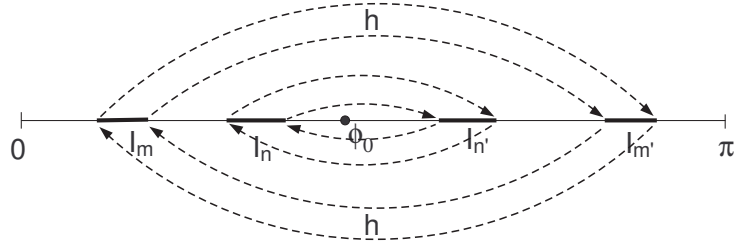}
    \caption{Dual intervals about the fixed point $\varphi_0$.}
    \label{FigInts}
\end{figure}

If the number of intervals $I_m$ is finite, then it is necessarily
even, and exactly half of them lies to the left of $\varphi_0$, and the
other half -- to the right. The above structure is essentially a
complete description of all interval maps $g$ satisfying
\eqref{gfixed}:

\begin{proposition}
Let $g$ be an orientation preserving homeomorphism of the interval
$[0,\pi]$ with an odd number of fixed points
$$
    0=\varphi_{-k}<\varphi_{-k+1}<\cdots<\varphi_0<\cdots
    \varphi_{k-1}<\varphi_{k}=\pi.
$$
Suppose that for each $0\leq i<k$ either all the points of both
intervals $(\varphi_{-i-1},\varphi_{-i})$ and
$(\varphi_{i},\varphi_{i+1})$ move under $g$ toward $\varphi_0$ or all
the points in these two intervals move away from $\varphi_0$. Then
there exists an orientation \emph{reversing} homeomorphism $h \colon
[0,\pi] \to [0,\pi]$ such that $g=h^2$. The middle fixed point
$\varphi_0$ is the (only) fixed point of $h$.
\end{proposition}

The proof uses standard methods of one-dimensional topological
dynamics, and we omit it.

The central fixed point $\varphi_0$ plays a special role, we will call
it \emph{the center} (of the map $g$); it is the only fixed point of
$h$. Note that
\beq  \label{fff}
    f_0(\varphi_0) = \pi-\varphi_0,
\eeq
hence if a particle hits a wall at the angle $\varphi_0$, it turns
around and flies straight back. Its trajectory is then periodic not
only in the angular coordinates, but also in the spatial coordinates;
see Fig.~\ref{FigBack}.

\begin{figure}[htb]
    \centering
   \includegraphics[width=2in]{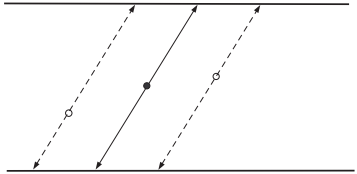}
    \caption{A periodic trajectory running between two collisions.}
    \label{FigBack}
\end{figure}

\section{A special family of collisions with a twist}

While most of our conclusions apply to generic functions $f_0$ and
$f_1$ satisfying \eqref{ffopp}, it will be convenient to use one
particular family of functions to clarify our arguments.

First, it is convenient when the center $\varphi_0$ is at the geometric
center of the interval $[0,\pi]$, i.e., $\varphi_0=\pi/2$. Then the
special periodic trajectories described in the end of
Section~\ref{SecOPC} move vertically, up and down, i.e., they just
bounce between the walls like the regular periodic billiard
trajectories.

Next, for simplicity we consider functions $g$ that
have one fixed point $\varphi_0 = \pi/2$ (other than
$0$ and $\pi$), which will automatically be the
center. Then there are three fixed points total:
$0,\pi/2$, and $\pi$. The intervals $(0,\pi/2)$ and
$(\pi/2,\pi)$ either both move toward $\pi/2$ or both
move away from $\pi/2$. In the former case $\pi/2$
will be a stable fixed point and attract the entire
interval $(0,\pi)$. In the latter case $\pi/2$ will be
an unstable (repulsive) fixed point and then each
interval $(0,\pi/2)$ and $(\pi/2,\pi)$ will be
attracted by its other endpoint, i.e., by $0$ or
$\pi$, respectively. We will say that in the former
case the map $g$ has a \emph{stable center} and in the
latter -- an \emph{unstable center}. These cases are
quite different, in dynamical terms, and both are
interesting.

We note that in order to make $\pi/2$ a fixed point for $g$ satisfying
\eqref{gfixed}, we need to set $f_0(\pi/2)=\pi/2$, according to
\eqref{fff}. Then $\pi/2$ is a fixed point for $f_0$ as well.

Thus $f_0$ must have three fixed points: $0$, $\pi/2$, and $\pi$. It is
also convenient if the time reversal of the collision rule defined by
$f_0$ belongs to the same family of collision rules. The time reversal
collision rule is defined by $\varphi\mapsto f_0^-(\varphi)$ at the
bottom wall and $\varphi \mapsto f^-_1(\varphi)$ at the top wall, where
the functions $f_0^-$ and $f_1^-$ satisfy
\beq \label{trrules}
    f_0^- = f^{-1}_1
    \qquad\text{and}\qquad
    f_1^- = f^{-1}_0
\eeq
in accordance with \eqref{ffopp}.

In view of the above requirement we can define $f_0$ by
\beq  \label{f0cot}
     \tan f_0(\varphi) = e^\lambda \tan \varphi.
\eeq
Then, due to \eqref{ffopp},
\beq  \label{f1cot}
     f_1(\varphi) = f_0(\varphi)
\eeq
so both walls obey the same collision rule! Note that there are no
restriction on $\lambda$ here, our rules work well for any $\lambda \in
(-\infty, \infty)$, but for the twist to be small we will only consider
$\lambda \approx 0$.

Due to \eqref{trrules}, the time reversal collision rules satisfy
$$
    \tan f_0^-(\varphi) = \tan f_1^-(\varphi) = e^{-\lambda} \tan \varphi,
$$
hence it belongs to the same family of rules \eqref{f0cot}, but
$\lambda$ must be replaced with $-\lambda$, i.e., reversing the time
corresponds to negating $\lambda$.

A direct differentiation of \eqref{f0cot} gives
$$
   f_0'(\varphi) = e^{-\lambda}\,\frac{\sin^2f_0(\varphi)}{\sin^2\varphi}
   =e^{\lambda}\,\frac{\cos^2f_0(\varphi)}{\cos^2\varphi},
$$
Hence
$$
   f_0'(\pi/2) = e^{-\lambda}
    \qquad\text{and}\qquad
   f_0'(0)=f_0'(\pi) = e^{\lambda}.
$$
Thus, $\lambda>0$ corresponds to a stable center and $\lambda<0$ to an
unstable center, i.e., we have dynamics of both types.

We note that the rules \eqref{f0cot}--\eqref{f1cot} can be rewritten as
follows: at every collision with the wall the incoming velocity
$(u^-,v^-)$ and the outgoing velocity $(u^+,v^+)$ are related by
\beq  \label{vvvv}
         u^+/|v^+| = e^{-\lambda} u^-/|v^-|.
\eeq
This makes it convenient for numerical simulations: one can recompute
the velocity vectors by \eqref{vvvv} without using the angle $\varphi$
or its tangent.

\section{Collisions with a stable center}

We begin with a single disk and a small positive $\lambda
>0$. Again, we assume that the disk has zero radius, i.e., it
is just a point particle.

Recall that $f_0'(\pi/2)=e^{-\lambda}<1$. Thus if the angle of
incidence is close to $\pi/2$, i.e., $\varphi = \pi/2 + \delta$ for
some small $\delta$, then after $n$ collisions at the walls it will be
$\varphi_n = \pi/2 + \delta_n$, where $\delta_n\to 0$.

More precisely, as the disk moves between collisions from wall to wall,
its displacement in the horizontal direction is given by $\Delta x =
\cot \varphi = \tan \delta$. Due to \eqref{f0cot}--\eqref{f1cot}, after
the next collision its displacement will be $e^{-\lambda} \cot\varphi =
e^{-\lambda}\, \Delta x$. Thus the displacement in the $x$ direction is
precisely decreasing by a factor $e^{-\lambda}<1$.

Thus the displacements in the $x$ direction make a geometric
progression and the total displacement is
\beq \label{sumx}
   \Delta x \sum_{n=0}^{\infty}e^{-n\lambda}  = \frac{\Delta x}{1-e^{-\lambda}}
\eeq
The particle's trajectory converges to a vertical line and the velocity
vector aligns vertically at an exponential rate.

Consider now the system of $N$ disks. We assume that the diameter $d$
of the disks is small enough so that the disks can be lined up along
one vertical or horizontal line in the unit square, i.e., $d<1/N$.

For any disk with center at $(x_i,y_i)$ and velocity vector $(u_i,v_i)$
denote by $s_i$ the total displacement in the $x$ direction of that
disk if it is allowed to bounce between the walls alone, as if the
other disks did not exist. By \eqref{sumx}, $s_i$ is finite, and in
fact
\beq  \label{Si}
   s_i \sim \frac{|u|}{|v|(1-e^{-\lambda})}
\eeq

Let $U$ be the the subset in phase space satisfying the following
condition:
\beq \label{U}
    |x_i-x_j|>d+s_i+s_j
    \qquad \forall i\neq j
\eeq
(note that the $x$ coordinates must be taken modulo 1, as we have
periodic boundary conditions at $x=0$ and $x=1$).

Now \eqref{U} guarantees that the projections of the disks onto the
$x$-axis do not overlap and will not overlap at any time in the future.
Thus our $N$ disks will never collide with each other, all of them will
be bouncing between the walls and the trajectory of each disk will
converge to some vertical line.

\begin{figure}[htb]
    \centering
   \includegraphics[width=3in]{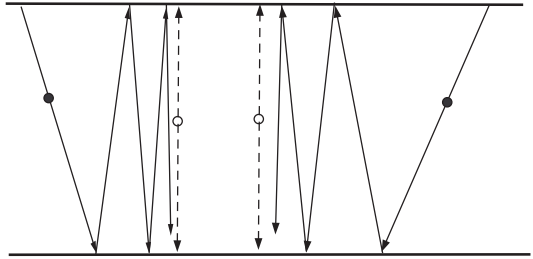}
    \caption{Convergence the a stable limit regime.}
    \label{FigStabN}
\end{figure}

In this limit regime, the $x$-component of the velocity vector of each
disk is zero. Thus those limit trajectories form a family of
codimension $N$ in phase space. Note that the phase space of our system
is $(4N-1)$-dimensional (the only constraint is the conservation of the
total kinetic energy). Thus our family of stable limit trajectories is
a $(3N-1)$-dimensional submanifold in the phase space. We denote it by
$S$.

Every phase trajectory starting in $U$ converges to $S$ at an
exponential rate. The set $U$ is invariant, i.e., all the trajectories
originating in $U$ stay there in the future. Note that $U$ is an open
subset of the phase space, thus it has a positive Lebesgue measure.
Therefore the limit set $S$ has an open basin of attraction of positive
Lebesgue measure.

It is more convenient to work with the collision space $\cM$ which
consists of all phase states such that either some two disks collide or
a disk collides with a wall. (At the moment of collision, the velocity
vectors change discontinuously, and it is customary to include in $\cM$
only the postcollisional velocity vectors). The induced map
$T=T_{\lambda} \colon \cM\to \cM$ is called the collision map.

We denote by $\nu_0$ the normalized measure on $\cM$ that is invariant
under the map $T_0$ corresponding to $\lambda=0$. This map corresponds
to the classical specular reflections at the walls (where
$\psi=\varphi$), so $T_0$ is the collision map for the classical gas of
hard disks at equilibrium. The measure $\nu_0$ is absolutely continuous
with respect to the Lebesgue measure on $\cM$, and it has a strictly
positive density. At the collisions with the walls, the density is
proportional to $\sin \varphi$ \cite{Ch97}.

Let $U_0 = U\cap \cM$ denote the subset of the collision space where
\eqref{U} holds. It is invariant under $T$, i.e., $T(U_0) \subset U_0$.
Its trajectories converge to the set $S_0 = S\cap \cM$, i.e., $U_0$ is
the basin of attraction of $S_0$ under the map $T$. The basin of
attraction $U_0$ may be regarded as a trap, or a ``hole'' -- every
phase trajectory that enters it will never come back.

\section{Holes and escape rates} \label{Holes}

Dynamical systems with holes have been studied extensively in the past
decades, both numerically and theoretically \cite{BBS,CMT,KL,PY,STN}. In
particular, many researchers studied billiards with holes
\cite{BLD,BKT,DWY,LM,OT}. It is generally observed (and in many cases
proved) that holes attract almost the entire phase space, i.e., almost
every trajectory sooner or later enters the hole and never returns (in
other words, it escapes). This phenomenon is also interpreted as the
leakage of mass (phase volume) through holes so that the remaining phase
space gets thinner.

The holes are usually characterized by the \emph{escape rate} $\rho>0$,
which basically says that the fraction of phase space that has
\emph{not} escaped through the holes (not leaked out) by the time $n>0$
(where $n$ is discrete time, i.e., the collision counter) decays
exponentially at a rate $\rho$, i.e., it is of order $e^{-\rho n}$.
More precisely, if $\cM_n$ denotes the subset of the phase space $\cM$
that has not escaped through the holes by the time $n$, then
$$
   \lim_{n\to\infty} \tfrac 1n\, \ln{\rm Leb}(\cM_n) = -\rho.
$$
For any phase state $X\in \cM$ we denote by $\tau(X)$ the \emph{escape
time}, i.e., the time the trajectory of $X$ enters the hole. Then
$\tau(X)$ has an approximate exponential distribution on $\cM$ with
parameter $\rho$ (with respect to the Lebesgue measure), and its
average is $1/\rho$.

Many studies investigate how the escape through holes is affected by
the size of the latter. When holes get smaller (shrink), the escape
rate decreases and the escape time grows. In our case, the ``hole''
$U_0$ described above depends on the parameter $\lambda>0$. When
$\lambda$ decreases, the hole $U_0$ shrinks and in the limit
$\lambda\to 0$ the hole converges to the submanifold $S_0 \subset \cM$
of zero volume.

We have investigated the escape process in our model by numerical
simulations. We used $N=2$ disks of diameter $d=0.1$ with the collision
rules at the walls defined by \eqref{f0cot}--\eqref{f1cot} for various
small $\lambda>0$. We have estimated the ``escape time'' numerically,
for different $\lambda$'s. For each $\lambda$ we simulated the dynamics
from $10^5$ randomly chosen initial states, stopping the simulations
whenever the system escaped into the ``hole'' $U_0$, and computed the
mean escape time $\mu_{\tau}$.

\begin{figure}[htb]
    \centering
   \includegraphics[width=4in]{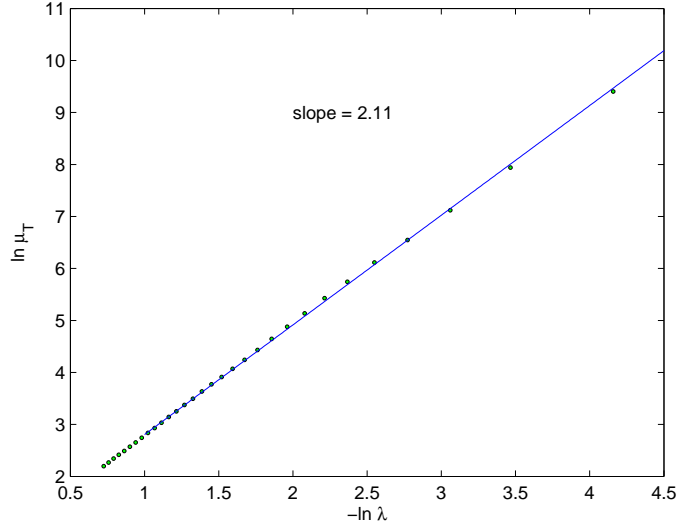}
    \caption{Plot of $\ln\mu_\tau$ versus $-\ln \lambda$. The
    least squares line has an estimated slope of 2.11.}
    \label{Figloglog}
\end{figure}

It naturally grows, as $\lambda$ gets smaller, and Figure~\ref{Figloglog}
shows how its grows on the log-log scale. The plot clearly demonstrates a
linear pattern suggesting that $\mu_\tau \sim \lambda^{-b}$ for some
$b>0$. The least squares fitting line in Figure~\ref{Figloglog} has slope
2.11, so we may guess that $b \approx 2$. Below we give a heuristic
argument supporting (and refining) this conjecture.

Recall that the ``hole'' $U_0\subset \cM$ is an open set of positive
Lebesgue measure, and it is invariant under the collision map $T$,
i.e., $T(U_0) \subset U_0$. Hence the sets
$$
   U_1=T^{-1}(U_0)\setminus U_0
   \qquad\text{and}\qquad
   U_k = T^{-k+1}(U_1)
$$
for $k=2,3,\ldots$ are disjoint and each of them has
positive Lebesgue measure, too. Their union
$U_{\infty} = \cup_{k=0}^{\infty} U_{k}$ contains all
the phase points $X\in \cM$ that end up in the hole
$U_0$ eventually. Our first conjecture (well supported
by our numerical observations) is that almost every
phase point $X\in \cM$ eventually ends up in the hole,
i.e., $U_{\infty}$ covers the entire collision space
$\cM$ (up to a null set), i.e., $\cM = U_{\infty}$
(mod 0). By direct analysis (we omit details)
\beq  \label{lambdaN}
    {\rm Leb}(U_0) = \cO(\lambda^N)
    \qquad\text{and}\qquad
   {\rm Leb}(U_{1}) = \cO(\lambda^{N+1}).
\eeq
and similar estimates hold for $\nu_0(U_0)$ and $\nu_0(U_1)$.

Next, the measure $\nu_0$ is invariant under the original, twist-free
collision map $T_0$, but not for $\lambda\neq 0$. In fact, it is
preserved by the interparticle collisions, but at collisions with the
walls it gets compressed if $f_0'(\varphi)<1$ and gets expanded if
$f_0'(\varphi)>1$. In our case, $f_0'(\varphi)<1$ whenever $c_1 <
\varphi < c_2$ for some constants $0< c_1 < c_2 < \pi$ whose values are
not essential.

For any point $X\in \cM$ we denote the Jacobian of the inverse
collision map $T^{-1}$, with respect to the canonical measure $\nu_0$
by $J(X)=e^{\gamma(X)}$. If $X$ corresponds to an inter-particle
collision, then $J(X)=1$, hence $\gamma(X)=0$. Otherwise $\gamma$ has a
small (of order $\lambda$) value.

Now for every point $X \in U_1$ and $k\geq 1$ we denote
$$
    J_k(X) = J(T^{-k+1}(X)) = e^{\gamma_k}, \qquad
    \gamma_k = \gamma(T^{-k+1}(X)),
$$
then we have
\beq \label{JJJ}
  \nu_0(U_{k+1}) = \int_{U_{1}} J_1\cdots J_k\,d\nu_0
   = \int_{U_{1}} e^{\gamma_1+\cdots+\gamma_k}\,d\nu_0.
\eeq

Next note that the sequence $\gamma_1,\ldots,\gamma_k$ is determined by
the collisions of the disks with the walls along the \emph{past}
trajectories of points $X\in U_{1}$. As time runs backwards, the disks
begin colliding with each other and a chaotic regime quickly sets in.
Then the distribution of every trajectory in the phase space will be
fairly close to uniform (equilibrium). Thus the values of $\gamma_n$
could be treated as nearly independent random variables. They are all
of order $\lambda$, hence the sum of $k$ of them can be expected to
grow as $\lambda\sqrt{k}$, in the spirit of the the Central Limit
Theorem. Thus by \eqref{JJJ} we can expect that
$$
  \nu_0(U_{k+1})
   \sim \int_{U_{1}} e^{\lambda\sqrt{k}}\,dX
   \sim \lambda^{N+1} e^{\lambda\sqrt{k}},
$$
as long as the chaotic regime continues. The obvious limitation $\nu_0
(U_{k+1}) \leq 1$ gives us an upper bound for $k$:
\beq   \label{kupper}
       k \leq \,{\rm const}\,\frac{(\ln\lambda)^2}{\lambda^2}
\eeq
Perhaps it is more reasonable to expect that $\nu_0 (U_{k+1}) \sim
1/k$, but this would give us pretty much the same upper bound
\eqref{kupper} on $k$.

Thus we see that the chaotic regime, when the particles collide and the
measure of the regions $U_{k}$ tends to grow, lasts for about
$\cO\bigl((\ln\lambda)^2/\lambda^2\bigr)$ collisions. As a result, the
mean `escape time' is of the same order:
\beq \label{tescape}
      \mu_\tau = \cO\bigl((\ln\lambda)^2/\lambda^2\bigr).
\eeq
On the log-log scale $x=-\ln\lambda$ and $y=\ln\mu_\tau$ adopted in
Figure~\ref{Figloglog} this means
\beq  \label{22}
   y = 2x + 2\ln x +\,{\rm const}.
\eeq
Accordingly, we used a functional relation $y=ax+b\ln x+c$ to
approximate our data plotted in Figure~\ref{Figloglog}, and the least
squares fit gives
$$
   y = 1.985x + 0.233\ln x + 0.811.
$$
The estimated slope of 1.985 is in a good agreement with the
theoretically predicted slope of 2. The logarithmic coefficient 0.233
is not close to 2 in \eqref{22}, but it is of secondary importance for
the estimate \eqref{tescape}.

We now investigate the dynamics beyond the bound \eqref{kupper}. The
measure $\mu_0 (U_{k})$ cannot increase with $k$ forever; in fact we
must have
$$
   \mu_0(U_{k}) \to 0\qquad \text{as}\quad k\to\infty,
$$
and moreover, $\mu_0(U_{k})$ should decay fast because the series
$\sum_k \mu_0(U_{k})$ is summable.

Thus for large $k$ (those far exceeding the upper bound \eqref{kupper})
the dynamics cannot be chaotic: it is dominated by the collisions with the
walls that compress the phase volume (when the time runs backwards). Those
collisions have incidence angles close to 0 or $\pi$, so that the
particles hit the walls nearly tangentially. This means that the particles
move nearly horizontally in the channel.

We recall that running the time backward for a dynamics with parameter
$\lambda$ is equivalent to reversing the velocity vectors of the disks
and running the time forward for the dynamics with parameter
$-\lambda$.

This suggests that for the dynamics with negative parameter $\lambda<0$
the limiting regime also exists, and it involves the disks flying
almost horizontally and experiencing nearly tangential collisions with
the walls. This will be established, with some rigor, in the next
sections.

\section{Collisions with an unstable center} \label{CUS}

Here we analyze the dynamics of $N$ disks with a negative parameter,
$\lambda<0$. In this case $\varphi = 0$ and $\varphi=\pi$ are stable
fixed points for both functions $f_0$ and $f_1$ and the derivative at
these points is
$$
   f_0'(0)=f_0'(\pi)=e^\lambda<1
$$
Thus if the angle of incidence is close to 0 or $\pi$, i.e., $\varphi =
\delta$ or $\varphi=\pi-\delta$ for some small $\delta$, then after a
collision at the wall it will be even closer to 0 or $\pi$, i.e., it
will be $\varphi_1=\delta_1$ or $\varphi_1=\pi-\delta_1$, respectively,
with $\delta_1 < \delta e^{\lambda'}$ with some constant $\lambda'<0$
($\lambda' \approx \lambda$ for small $\delta$).

However the angles of incidence cannot converge to 0 or $\pi$
monotonically in a manner similar to their convergence to $\pi/2$ in
the previous section. Indeed, after a nearly tangential collision with
the wall the disk has to move across the channel to the opposite wall,
and now on its way across it will be very likely colliding with other
disks.

\begin{lemma}  \label{lmccc}
With probability one the disks will continue colliding at arbitrary
distant future.
\end{lemma}

\begin{proof}
Suppose that the disks never collide with each other after some time
$T>0$. Then they just collide with the walls, and hence their velocity
vectors align almost horizontally and converge to some horizontal
vectors.

By simple geometry, the disks can avoid each other on their way from
one wall to the other only if the horizontal components of their
velocities are nearly equal. And as the vertical components get
smaller, the horizontal components would have to get also closer to
each other. So in the limit, as time grows to infinity, the velocity
vectors of the disks would have to converge to a common limit!

However, the collisions with the walls do not alter the kinetic energy
of the disks. Thus if the disks somehow managed to avoid each other,
the speed of each disk would remain constant. And since their velocity
vectors must have a common limit, their speeds must have been equal all
the time! In other words, the `equal speeds' situation should have
occurred right after the last collision between the disks in the past.
This event is exceptional and occurs with probability zero.
\end{proof}

Thus almost surely our disks will keep colliding with each other
forever. For this reason the limiting stable regime (if one exists)
could not be as simple as the one we have seen in the previous section.
Still a limiting stable regime exists, as we will show next. Again we
assume that the disks are not too large.

Let us set the total kinetic energy to $N/2$, so that the velocity
vectors $(u_i,v_i)$ of the disks will satisfy
\beq  \label{Etot}
   u_1^2+v_1^2+\cdots+u_N^2+v_N^2 = N.
\eeq
Now we define a subset $W_+$ in phase space by
\beq  \label{W+}
       W_+=\{u_1+\cdots+u_N>\sqrt{N(N-1)} \}.
\eeq
An elementary calculation shows that for every $X\in W_+$
\beq  \label{u>0}
     u_i>0 \qquad \text{for all} \quad i=1,\ldots,N.
\eeq
Now consider the trajectory of a phase point $X\in W_+$. At every
collision with a wall, the incoming velocity vector $(u_i,v_i)$ points
to the right, due to \eqref{u>0}. Then the postcollisional velocity
vector $(u_i',v_i')$ will turn toward the wall, i.e., $u_i'>u_i$ and
$|v_i'|<|v_i|$. At every collision between disks $i$ and $j$, the
velocity vectors of both disks may change, but their total momentum is
preserved, i.e., $u_i+u_j$ remains the same.

We see that the total horizontal momentum of the system
\beq  \label{Mu}
       M_u = u_1+\cdots+u_N
\eeq
is monotonically increasing: it is preserved at every interparticle
collision and increases at every collision with a wall. Thus the
trajectory of every point $X\in W_+$ will remain in $W_+$ forever.
Furthermore, the value of $M_u$ will grow and converge to a limit.

\begin{figure}[htb]
    \centering
   \includegraphics[width=4in]{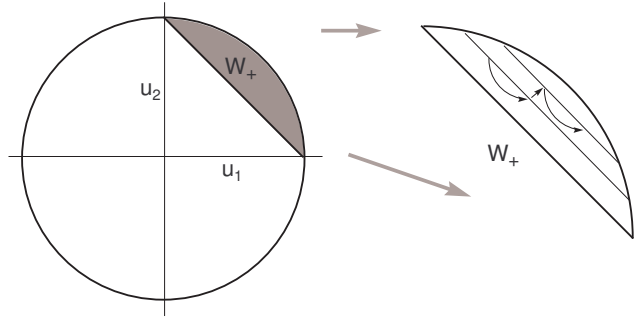}
    \caption{The region $W_+$ and the dynamics inside it.}
    \label{Figuu}
\end{figure}

We illustrate this process for $N=2$. Figure~\ref{Figuu} shows $W_+$ in
the $u_1,u_2$ coordinates: $W_+$ is bounded by the circular arc
$u_1^2+u_2^2=2$, $u_1,u_2 > 0$, and its chord $u_1+u_2=\sqrt{2}$. Then
$W_+$ is foliated by parallel lines $u_1+u_2=c$, $\sqrt{2}<c\leq 2$.
The point $(u_1,u_2)$ stays on the same line $u_1+u_2=c$ at every
collision between the disks and moves up to another line $u_1+u_2=c'$
with $c'>c$ during every collision with a wall.

Next we investigate the limit regime. Clearly there is a limit chord
$u_1+u_2=c_{\infty}\leq 2$. If $c_{\infty}=2$, the chord degenerates to
a single point $u_1=u_2=1$. In that case $v_1=v_2=0$, so the limit
regime consists of both disks moving horizontally with the same unit
speed.

If $c_{\infty}<2$, then the points $(u_1,u_2)$ along
the trajectory of $X\in W_+$ accumulate on the chord
$u_1+u_2=c_{\infty}$. We claim that they actually
converge to one of the end points of that chord.
Indeed, when the point $(u_1,u_2)$ is at a distance
$\delta_1>0$ from the end points of the current chord
$u_1+u_2=c$, then $u_1^2+u_2^2<2-\delta_2$ for some
$\delta_2>0$ (which depends on $\delta_1$). Therefore
$|v_1|>\delta_3$ or $|v_2|>\delta_3$ for some
$\delta_3>0$ (which depends on $\delta_2$). Now we
claim that at the next collision with a wall by one of
the disks its vertical velocity will be
$|v_i|>\delta_4$ for some $\delta_4>0$ (which depends
on $\delta_3$). Indeed, a sequence of consecutive
collisions between the two disks can be reduced to a
simple billiard trajectory in a 2D periodic Lorentz
gas \cite[Section~4.2]{CM06}, and then one can easily
check that the vertical components of their velocities
cannot both vanish at the time they hit the walls.

Next, when a disk with a vertical velocity
$|v_i|>\delta_4$ hits a wall, its velocity vector will
be rotated so that $u_i$ will grow by some
$\delta_5>0$ (which depends on $\delta_4$). Thus the
point $(u_1,u_2)$ will move up to a chord
$u_1+u_2=c+\delta_6$ for some $\delta_6>0$ (which
depends on $\delta_5$). Hence the convergence to a
limit chord can only occur when $(u_1,u_2)$ converges
to one of its end points (and, as a result, $v_1,v_2$
converge to zero).

In the above limit regime, when $c_{\infty}<2$, both disks move
horizontally but at different speeds, $\bar{u}_1\neq \bar{u}_2$. We
believe this actually happens with probability zero, but we can only
give a heuristic argument. By Lemma~\ref{lmccc} the disks have to
collide from time to time. When they collide, their relative horizontal
velocity is $u_1-u_2 \approx \bar{u}_1-\bar{u}_2\neq 0$, but their
relative vertical velocity is $v_1-v_2= \varepsilon \approx 0$. The
collision of two disks can be reduced to a 2D periodic Lorentz gas with
infinite horizon \cite[Sect. 4.2]{CM06}.

\begin{figure}[htb]
    \centering
   \includegraphics[width=4in]{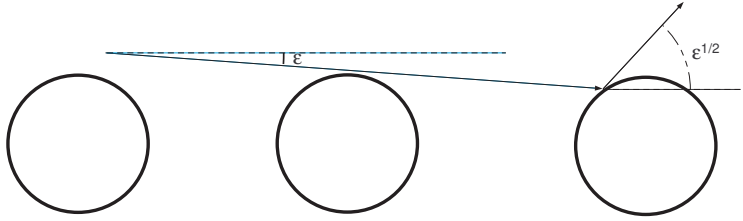}
    \caption{Nearly horizontal trajectories in a periodic Lorentz gas.}
    \label{FigLG}
\end{figure}

It is known in the studies of periodic Lorentz gases that if the
vertical velocity of the moving particle before the collision is
$\varepsilon$, then after the collision it is typically of order
$\sqrt{\varepsilon}$, i.e., $v_1'-v_2'=\cO(\sqrt{\varepsilon})$; see
Figure~\ref{FigLG}. More precisely, there are $p,q>0$ such that the
relative measure of initial conditions for which
$v_1'-v_2'=\cO(\varepsilon^{1/2-p})$ is $\cO(\varepsilon^q)$; see
\cite{SV}. Thus in the course of infinitely many successive collisions
between the disks their vertical velocities would explode sooner or
later with probability one, violating the assumption $c_{\infty}<2$.
This argument is not quite formal, as the word ``probability'' here
refers to the canonical measure $\nu_0$, and in our dynamics the images
of this measure keep changing, but we believe the conclusion is
correct.

To summarize, we proved that in the limit regime the particles move
horizontally to the right. We also conjecture that with probability one
their horizontal velocities are equal:
\beq \label{uuvv}
   u_1=\cdots=u_N=1
         \qquad\text{and}\qquad
   v_1=\cdots=v_N=0.
\eeq
In other words, almost all trajectories in $W_+$ converge to a limit
regime where all the particles move horizontally and at unit speed.
This is a stable limit regime, it corresponds to a submanifold $S_+$.
This submanifold is defined by equations \eqref{uuvv}, hence its
dimension is $2N$ (corresponding to the free coordinates $(x_i,y_i)$ of
all the particles).

There is a symmetric region $W_-$ in phase space defined by
\beq  \label{W-}
       W_-=\{u_1+\cdots+u_N<-\sqrt{N(N-1)}\},
\eeq
where all the particles move in the negative $x$ direction. By a
similar argument, almost all trajectories $X\in W_-$ converge to the
stable limit regime where all the particles move horizontally to the
left, and (we again conjecture) at unit speed:
\beq \label{uuvv-}
   u_1=\cdots=u_N=-1
         \qquad\text{and}\qquad
   v_1=\cdots=v_N=0.
\eeq
This corresponds to a submanifold $S_-$ of dimension $2N$.

Note that we have two stable limit regimes. Each one attracts a part of
$\cM$ that has a positive Lebesgue measure. Their basins of attraction
are obviously disjoint, and due to symmetry they must have the same
$\nu_0$-measure. Thus we have a coexistence of two attracting
mechanisms in phase space.

\begin{figure}[htb]
    \centering
   \includegraphics[width=4in]{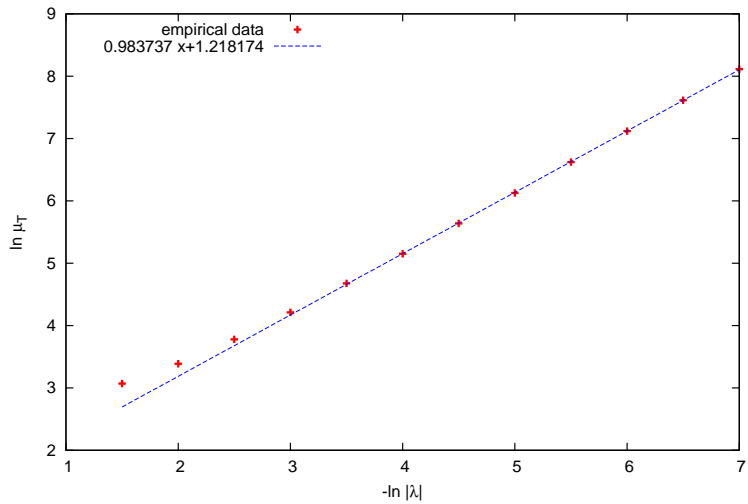}
    \caption{Plot of $\ln\mu_\tau$ versus $-\ln |\lambda|$. The
    least squares fit has slope of 0.983737.}
    \label{Figloglog1}
\end{figure}

Just like in Section~\ref{Holes} we estimated the escape time numerically.
We used $N=2$ disks of diameter $d=0.1$ with the collision rules at the
walls defined by \eqref{f0cot}--\eqref{f1cot} for various small
$\lambda<0$. Figure~\ref{Figloglog1} shows the mean escape time $\mu_\tau$
versus $\lambda$ on the log-log scale. The plot clearly demonstrates a
linear pattern and the least square fitting line has slope 0.984
suggesting that $\mu_\tau \sim |\lambda|^{-1}$. Thus typical trajectories
now escape much faster than in the case of stable center. We discuss the
reason for the faster escape below.

Our main argument in Section~\ref{Holes} used the central limit theorem
and was based on the chaotical character of the motion of the balls
(before they enter the stable regime). However when the balls collide with
the walls, our twisting rules \eqref{f0cot}--\eqref{f1cot} with
$\lambda<0$ make the horizontal components of their velocities increase,
thus the $x$ component of the total momentum, $u_1+u_2$, tends to grow (in
absolute value). This tendency causes a drift toward larger values of
$|u_1+u_2|$, i.e., toward the trapping regions $W_{\pm}$. Thus we should
regard $u_1+u_2$ as a one-dimensional It\^o diffusion process with a
non-zero drift \cite{O,RY}.

\begin{figure}[htb]
    \centering
   \includegraphics[width=4in]{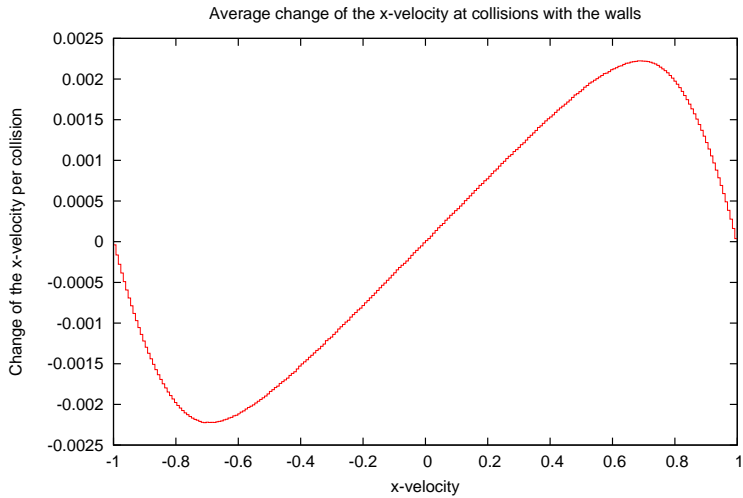}
    \caption{Average change of the $x$ component of the
disks velocity $u$ versus its current value.}
    \label{FigCM}
\end{figure}

Figure~\ref{FigCM} described the degree of the drift: it shows the average
change of $u$, the $x$ component of the disk velocity, versus its current
value (the plot was computed for $\lambda=-0.015$). We see that the
absolute value $|u|$ always tends to increase, and this tendency is strong
everywhere except when $u$ is close to 0 or $\pm 1$. In
Section~\ref{Holes} we dealt with the balls that were approaching the hole
$U_0$; then they moved nearly vertically, hence the drift was small and
could be ignored. Now the balls are approaching the regions $W_{\pm}$;
thus the drift is almost at its peak and its effect is crucial. Since the
average drift, per collision, if of order $|\lambda|$, typical
trajectories reach the regions $W_{\pm}$ of order $|\lambda|^{-1}$.

Figure~\ref{FigCM} also shows that the drift persists as long as $|u|<1$,
though it decreases as $|u|\to 1$. This supports our conjecture that
typical trajectories converge to a limit regime where all the particles
move at the same unit speed.

\section{Generalizations}

To summarize, the dynamics with unstable center ($\lambda<0$) is quite
different from the one with stable center ($\lambda>0$). It has two
distinct limit regimes ($S_+$ and $S_-$), both have dimension $2N$,
while in the other case we had a single stable regime $S$ with
dimension $3N-1$.

The attracting regions $W_{\pm}$ (``holes'') are of fixed size and
measure (independent of $\lambda$), while for stable center the
``hole'' $U_0$ was quite narrow and had measure $\cO(\lambda^N)$.

Next we generalize our results. Our analysis of the stable center case
applies to any stable fixed point of the map $g= f_1\circ f_0$ inside the
interval $(0,\pi)$. Every such point produces a $T$-invariant open set
$U_0 \subset \cM$, i.e., $T(U_0)\subset U_0$, which acts as a ``hole''. It
has a basin of attraction in $\cM$ of positive Lebesgue measure, and the
limiting stable regime consists of disks moving with periodic incidence
angles (assuming that the disks are small enough to avoid colliding with
each other).

Of course if the map $g\colon [0,\pi] \to [0,\pi]$ has more than one
stable fixed point, the corresponding basins of attraction are
disjoint, so we have a coexistence of several attracting regimes.

Our analysis of the unstable center case applies to
any map $g$ with stable fixed end points, 0 and $\pi$.
In that case we have two attracting regimes similar to
$S_+$ and $S_-$ above, and each has its own basin of
attraction.

We must note that in the general case the ``holes'' $W_{\pm}$ need be
defined more cautiously than \eqref{W+} and \eqref{W-}. Precisely, we
need to choose a small $\varepsilon_0>0$ and define $W_{\pm}$ by
$$
       W_{\pm}=\{\pm(u_1+\cdots+u_N)>N-\varepsilon_0\}.
$$
It is easy to verify that there is a $c=c(\varepsilon_0)>0$ such that
for every phase state in $W_{\pm}$ we have
$$
         \pm u_i>1-c
         \qquad\text{and}\qquad
         |v_i|<c
$$
for all $i=1,\ldots,N$. Moreover, $c\to 0$ as $\varepsilon_0\to 0$. So
we can choose, for example, $c=0.1$ and fix the corresponding
$\varepsilon_0>0$ in the definition of $W_{\pm}$. Then of our analysis
of the convergence to $S_{\pm}$ will not require significant changes.

In summary, every stable fixed point of the map $g\colon [0,\pi] \to
[0,\pi]$ leads to a stable regime that attracts a set of positive
Lebesgue measure in phase space.

In our numerical experiments, we mostly observed stable regimes for
$N=2$ disks. Our analysis shows that they exist for larger $N$'s, too,
but they become increasingly difficult to observe experimentally.
Indeed, the sizes of the holes in phase space are exponentially small
in $N$ (cf.\ \eqref{lambdaN}, and a similar estimate can be derived for
\eqref{W+}), hence the escape time grows exponentially with $N$. Thus
in physical systems with a large number of molecules the stable regimes
become almost unobservable. Still, they exist for any $N$, and they
dominate the dynamics for small $N$'s.

One may wonder whether attracting regimes exist when $g$ has no stable
fixed points. For example, what if $g(\varphi)=\varphi$ is the identity
map? The next section presents the most striking result: even in that
case stable regimes may exist, which attract almost every phase
trajectory!

\section{Time reversible collisions with walls}

Of special physical interest are twisting collisions
\eqref{f} that make the dynamics \emph{time
reversible}. This means that if we reverse the
velocity vectors of all our disks, then they will move
along their past trajectories backwards. By direct
inspection, the collision rule \eqref{f} is time
reversible if and only if the function $f$ satisfies
\beq \label{symm}
   f\bigl(\pi-f(\varphi)\bigr)=\pi-\varphi.
\eeq
This condition implies that the graph of the function $\psi=f(\varphi)$
is symmetric about the line $\psi=\pi-\varphi$.

If our collision rule is time reversible, i.e., satisfies \eqref{symm},
then we immediately arrive at $g(\varphi) = \varphi$.

\begin{figure}[htb]
    \centering
   \includegraphics[width=3in]{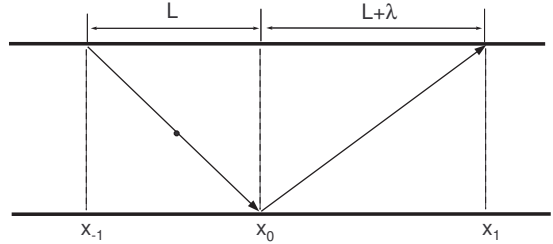}
    \caption{Collision rule pushing the particles to the right.}
    \label{FigV}
\end{figure}

A simple example of a time-reversible rule is
\beq  \label{fcot0k}
    \cot f_k(\varphi) = (-1)^k\lambda + \cot\varphi
\eeq
One can easily see that this definition is consistent with
\eqref{ffopp} and \eqref{symm}. The rule \eqref{fcot0k} can be written
in the notation of \eqref{vvvv} as
\beq  \label{vvuu}
         u^+/|v^+| = u^-/|v^-| + (-1)^k\lambda.
\eeq
This formula has a simple geometric interpretation: suppose a particle
collides with a wall at a point with $x$-coordinate $x_0$, and we
extend its trajectories before and after the collision until they cross
the other wall at points whose $x$-coordinates we denote by $x_{-1}$
and $x_1$, respectively (see Figure~\ref{FigV}), then
$$
    x_1-x_0 = x_0-x_{-1} + (-1)^k\lambda.
$$
If, for example, $\lambda>0$, then the above relation implies that the
bottom wall $y=0$ pushes the particles to the right, and the top wall
-- to the left. This creates shear flow in the channel of the type
studied in \cite{CL97}.

Next we describe a special regime in the above shear flow for $N=2$
disks. Suppose the disks move with opposite velocity vectors (i.e.,
their total momentum is zero) and they collide with opposite walls
simultaneously. Then after the collision they again move with opposite
velocity vectors.

When the disks collide with each other, then by symmetry their point of
contact has coordinate $y=0.5$, i.e., the collision occurs right in the
center of the channel. After the collision the disks move with opposite
velocity vectors again, and will collide with the opposite walls
simultaneously, etc. Hence this regime is invariant under our dynamics.
See illustration in Figure~\ref{FigSymm}.

\begin{figure}[htb]
    \centering
   \includegraphics[width=4in]{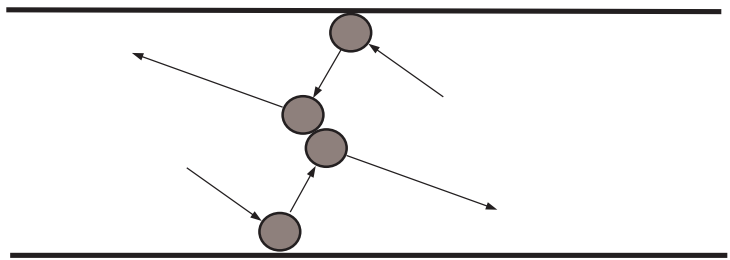}
    \caption{Special regime $S^\ast$ with symmetric motion.}
    \label{FigSymm}
\end{figure}

We denote the above special family of phase trajectories by $S^\ast$.
Note that dim$\,S^\ast = 3$, while the (reduced) phase space is
6-dimensional, cf.\ Section~\ref{Model}.

A striking fact we discovered by numerical simulations is that the
regime $S^\ast$ is stable and attracting. Almost every randomly
selected phase trajectory eventually stabilizes near $S^\ast$ and then
evolves in a symmetric fashion so that the disks move with opposite
velocity vectors at equal distances from the opposite walls.

Actually, the regime $S^\ast$ may be also stable for the dynamics with
an unstable center (Section~\ref{CUS}), but this happens only for
relatively large perturbations ($\lambda>0.6$). Since we are primarily
interested in small perturbations, we will not discuss this last fact.

We provide a semi-heuristic argument showing that the regime $S^\ast$
is stable for the time-reversible collision rules, i.e., phase
trajectories near $S^\ast$ tend to get closer to $S^\ast$ in the
future.

We perturb the symmetries of the regime $S^\ast$ and show that
perturbations tend to decrease, on average. There are two symmetries in
the regime $S^\ast$: the velocity vectors of the disks are opposite,
$\bv$ and $-\bv$, and their $y$-coordinates sum up to one, i.e.,
$y_1+y_2=1$.

First we perturb the velocity symmetry, i.e., suppose that the velocity
vectors are $\bv+\delta\bv$ and $-\bv$, i.e., the total momentum is a
small $\delta\bv \neq 0$. The interparticle collisions do not change
the total momentum. When the disks collide with (opposite) walls, then
the postcollisional velocity vectors will be denoted by
$\bv'+\delta\bv'$ and $-\bv'$. Our goal is to show that $|\delta \bv'|$
tends to be smaller than $|\delta\bv|$, on average.

We decompose $\delta\bv = \delta\bv_{\|} + \delta\bv_\perp$ into the
components parallel and perpendicular to $\bv$, respectively.
Similarly, $\delta\bv' = \delta\bv_{\|}' + \delta\bv_\perp'$ are the
components of $\delta\bv'$ parallel and perpendicular to $\bv'$. Due to
the conservation of the kinetic energy of each disk at the collision
with the wall we have
$$
      |\bv'|=|\bv|
      \qquad\text{and}\qquad
      |\bv'+\delta\bv'|=|\bv+\delta\bv|
$$
hence
\beq  \label{vv||}
   |\delta \bv_{\|}'| = |\delta \bv_{\|}|,
\eeq
to the leading order. We will show that $|\delta \bv_{\perp}'| <
|\delta \bv_{\perp}|$, on average.

Let $\varphi$ and $\varphi + \delta\varphi$ denote the directional
angles (i.e., angles made with the positive $x$ axis) of the velocity
vectors $\bv$ and $\bv+ \delta \bv$, respectively. Let $\psi$ and $\psi
+ \delta\psi$ denote the directional angles of the velocity vectors
$\bv'$ and $\bv' + \delta \bv'$, respectively. Then
$$
   \frac{|\delta \bv_{\perp}|}{|\bv|} = |\delta \varphi|
   \qquad \text{and} \qquad
   \frac{|\delta \bv_{\perp}'|}{|\bv'|} = |\delta \psi|
$$
to the leading order, hence
\beq  \label{dd}
   \frac{|\delta \bv_{\perp}'|}{|\delta \bv_{\perp}|}
   = \frac{|\delta \psi|}{|\delta\varphi|} = |f_k'(\varphi)| = f_k'(\varphi)
\eeq
(recall that $f_k$ is monotonically increasing, i.e., $f_k'>0$).

Now after the collisions with the walls the particles move across the
channel, and they either collide with each other or they miss each
other and hit the opposite walls. In the latter case we denote by
$\bv''+\delta\bv''$ and $-\bv''$ their new postcollisional velocity
vectors, and again decompose $\delta\bv'' = \delta\bv_{\|}'' + \delta
\bv_\perp''$. Then, inductively,
$$
  |\delta \bv_{\|}''| = |\delta \bv_{\|}'| = |\delta \bv_{\|}|
$$
and
$$
   \frac{|\delta \bv_{\perp}''|}{|\delta \bv_{\perp}|}
   = \frac{|\delta \bv_{\perp}''|}{|\delta \bv_{\perp}'|}
   \frac{|\delta \bv_{\perp}'|}{|\delta \bv_{\perp}|}
   = f_{1-k}'(\psi)f_k'(\varphi) = g'(\varphi)=1,
$$
therefore $|\delta\bv''| = |\delta\bv|$.

We see that as long as the particles collide with the walls only, the
perturbation vector $\delta \bv$ changes periodically, with period two.
So for an even number of wall collisions between two successive
interparticle collisions, the net result will be zero change, i.e.,
$|\delta \bv|$ will remain the same. For an odd number of wall
collisions between two successive interparticle collisions, the net
result will be the same as for just one wall collision, i.e.,
\eqref{dd} will apply. Next we relate $|\delta\bv_{\perp}|$ to
$|\delta\bv|$.

Consider an interparticle collision. It preserves the entire vector
$\delta \bv$, but the direction of the postcollisional velocity vector
$\bv$ can be regarded as a random variable uniformly distributed in the
entire range $[0,2\pi]$, due to the scattering nature of the elastic
collisions of hard disks. Let $\beta$ denote the directional angle of the
vector $\delta \bv$. Let $\varphi$ denote, as usual, the directional angle
of the outgoing velocity vector $\bv$. Then $|\delta\bv_{\perp}| =
|\delta\bv|\,|\sin(\beta-\varphi)|$. The perturbation $\delta \bv'$ at the
next interparticle collision will be
\beq  \label{ddlog1}
   |\delta \bv'|^2 = |\delta \bv|^2 \bigl(\cos^2(\beta-\varphi)
+ \varkappa \sin^2(\beta-\varphi)\bigr)
\eeq
where $\varkappa=[f_k']^2$ for an odd number of intermediate wall
collisions and $\varkappa=1$ for an even number of intermediate wall
collisions.

Next we estimate the total change of the norm $|\delta \bv|$ over a
long period of time $(0,T)$ during which $n$ interparticle collisions
occur. We have $n-1$ intervals between successive interparticle
collisions, and some of them (say, $m \leq n-1$ of them) have an odd
number of collisions of each particle with the walls, while others
(i.e., $n-m-1$ intervals) have an even number of collisions of each
particle with the walls. Our previous analysis can be summarized as
\beq  \label{ddlogm}
   \log|\delta \bv_T| - \log|\delta \bv_0|
   = \tfrac 12\,\sum_{i=1}^m \log\bigl[1-
   \bigl(1-[f_{k_i}'(\varphi_i)]^2\bigr)\sin^2(\beta_i-\varphi_i)\bigr],
\eeq
where the summation is taken over the intervals with odd numbers of
wall collisions.

Due to the randomization caused by the scattering effect of the
interparticle collisions we can treat the angles $\beta_i$ and
$\varphi_i$'s as independent random variables with uniform distribution in
their ranges $0\leq \beta_i<2\pi$ and $0\leq \varphi_i\leq\pi$. We will
prove in Appendix that the average value of each term in \eqref{ddlogm} is
negative:

\begin{lemma}  \label{LmI1}
For both $k=0,1$ we have
$$
  \int_0^{\pi}\int_0^{2\pi}\log\bigl[1-
   \bigl(1-[f_{k}'(\varphi)]^2\bigr)\sin^2(\beta-\varphi)
   \bigr]\, d\beta \, d\varphi =
   \mu_1<0.
$$
\end{lemma}

Thus the sum in \eqref{ddlogm} approaches $-\infty$ linearly in $m$
(i.e., linearly in time), hence the norm of the perturbation $\delta
\bv$ decreases exponentially in time.

Second, we perturb the other symmetry of the regime $S^\ast$, i.e., we
suppose that an interparticle collision occurs slightly above or below
the central line $y=0.5$ of the channel. More precisely, let the point
of contact at the moment of collision have coordinate $y=0.5+\ell$. At
the same time we suppose that the particles have opposite velocity
vectors, $\bv = (u,v)$ and $-\bv = (-u,-v)$, i.e., their total momentum
is zero.

When the particles collide with the opposite walls, their
postcollisional velocity vectors $(u',v')$ and $(-u',-v')$ will again
be opposite. But the particles collide with the walls at slightly
different moments of time, and the time interval between their
collisions with the walls will be $2|\ell/v|$. If, after those
collisions with the walls the particles collide with each other again,
the point of contact will have coordinate $y=0.5+\ell'$ with $|\ell'| =
|\ell v'/v|$, i.e.,
\beq  \label{ell1}
   \log |\ell'| - \log |\ell| = \log \bigl(|v'|/|v|\bigr)
   = \log \bigl(\sin f(\varphi)/ \sin\varphi\bigr).
\eeq

However, if the particles miss each other, then after their second
collision with the opposite walls their velocities will be again
$(u,v)$ and $(-u,-v)$. If the particles collide with each other after
that, the point of contact will be again the distance $|\ell'| =
|\ell|$ from the central line $y=0.5$.

Thus the parity issue again arises and the formula \eqref{ell1} applies
whenever the number of wall collisions between successive interparticle
collisions is odd; otherwise there is no change, $|\ell'| = |\ell|$.
Thus, in the notation of \eqref{ddlogm} we have
\beq  \label{ellm}
   \log|\ell_T| - \log|\ell_0|
   \sim \sum_{i=1}^m \log \bigl( \sin f_{k_i}(\varphi_i)/ \sin\varphi_i\bigr).
\eeq

Again we treat $\varphi_i$'s as independent uniformly distributed
random variables and verify that the mean value of $\log \bigl(\sin
f(\varphi)/ \sin\varphi\bigr)$ is negative. Proof of the following
lemma will be given in Appendix.

\begin{lemma}  \label{LmI2}
For both $k=0,1$ we have
$$
  \frac{1}{\pi}\,
  \int_0^{\pi} \log \bigl(\sin f_k(\varphi)/ \sin\varphi\bigr)\, d\varphi =
   \mu_2<0.
$$
\end{lemma}

Thus the sum in \eqref{ellm} approaches $-\infty$ linearly in $m$
(i.e., linearly in time), hence the magnitude of the positional
perturbation $\ell$ decreases exponentially in time.

This all verifies the stability of the regime $S^\ast$. Indeed, our
perturbations by $\delta\bv$ account for two directions transversal to
$S^\ast$, and those by $\ell$ account for one more, hence we took care of
all the three codimensions in the (reduced) 6-dimensional phase space.

We note that the classical (unperturbed) system of $N=2$ hard balls also
leaves the manifold $S^\ast$ invariant. The dynamics within $S^\ast$
easily reduce to a periodic Lorentz gas with a single particle and
infinite horizon. The periodic Lorentz gas is strongly hyperbolic and
ergodic, thus the phase points $X\in S^\ast$ have one positive and one
negative Lyapunov exponents, and typical phase trajectories within the
manifold $S^\ast$ fill it densely. Interestingly, there are no expansion
or contraction in any transversal direction to $S^\ast$, i.e., the points
$X\in S^\ast$ have exactly one positive and one negative exponents with
respect to the unperturbed dynamics in the entire phase space. This makes
the manifold $S^\ast$ exceptional, as typical phase points are proven
\cite{Sy99} to have two positive and two negative Lyapunov exponents, cf.\
Section~\ref{Model}.

We now describe what happens under our time-reversible perturbations.
First, all the (previously zero) Lyapunov exponents in the directions
transversal to $S^\ast$ seem to become negative, which instantly makes
$S^\ast$ an attractor. (Here we refer to Lyapunov exponents of
\emph{typical} points $X\in S^\ast$.) Second, the (previously equivalent
to the Lorentz gas) dynamics within $S^\ast$ is also perturbed, and it
seems to retain its hyperbolic character and admit an SRB measure. For
periodic Lorentz gases with infinite horizon under small perturbations the
hyperbolicity is proven and a (unique) SRB measure is constructed in
\cite{CD09}. Perhaps the arguments of \cite{CD09} could work in our case,
too.

As a result, the SRB measure living on $S^\ast$ seems to attract nearly
entire phase space. This makes it (the only) physically observable
measure, as it describes the distribution of typical phase trajectories.
Therefore, some kind of chaotic behavior does exist in the present case,
but only after the six-dimensional phase space is reduced to the 3D
surface $S^\ast$.

Lastly, we tried to find similar stable regimes for $N\geq 3$ balls with
time-reversible twists at the walls, and did not observed any -- the
dynamics seem to be totally chaotic.

\medskip \noindent\textbf{Acknowledgement.} This work was motivated by
discussions with J.~Lebowitz and Ya.~Sinai in March 2011, and we
acknowledge the hospitality of Rutgers University and Princeton
University. We thank D. Dolgopyat, J. Lebowitz, and A. Neishtadt for
many useful remarks. N.~C. is partially supported by NSF grant
DMS-0969187. We are also grateful to the Alabama supercomputer
administration for computational resources.

\section*{Appendix}

First we prove Lemma~\ref{LmI1}. We note that $f_k'(\varphi)>0$ and
$\int_0^\pi f_k'(\varphi) \, d\varphi=\pi$. With these assumptions the
integral over $\beta$ can be taken analytically. It results in
$$ \int_0^{2\pi}\log \left[ 1-\left(1-\left[f_k'(\varphi)\right]^2 \right)
\sin^2(\beta-\varphi)\right] \, d\beta =
4 \pi  \log\left[\frac{1}{2} \left(1+f_k'(\varphi)\right)\right] $$
Now Jensen's inequality works to show that
$$\frac{1}{\pi}\int_0^\pi \log\left[\frac{1}{2} \left(1+f_k'(\varphi)\right)\right] \,d\varphi
\leq \log \left[ \frac{1}{\pi}\int_0^\pi \log\left[\frac{1}{2} \left(1+f_k'(\varphi)\right)\right] \,d\varphi
 \right] = 0
$$
Lemma~\ref{LmI1} is proved. \qed

Now we prove Lemma~\ref{LmI2}. Our arguments apply to both functions
$f_0$ and $f_1$, so we suppress the index and denote them by $f$.

Recall that the graph of the function $\psi=f(\varphi)$ is symmetric
about the line $\psi=\pi-\varphi$; see Figure~\ref{FigGraph1}. Thus it
is convenient to use new variables
$$
    s = \frac{f(\varphi)+\varphi}{2}
    \qquad\text{and}\qquad
    w = \frac{f(\varphi)-\varphi}{2}.
$$
If we drop a perpendicular from the point $(\varphi, f(\varphi))$, that
lies on the graph, onto the main diagonal line $\psi=\varphi$, then its
footpoint will be $(s,s)$, and its length will be $w\sqrt{2}$. The
function $\psi=f(\varphi)$ becomes, in new variables, $w = w(s)$, and
due to the above symmetry we have $w(\pi-s) = w(s)$, i.e., $w$ is an
even function with respect to the center $s=\pi/2$. We also note that
$w(0)=w(\pi)=0$, $|w'|<1$, and
$$
   \varphi = s-w
    \qquad\text{and}\qquad
   f(\varphi) = s+w,
$$
therefore
$$
   d\varphi = (1-w')\, ds
    \qquad\text{and}\qquad
   \frac{df}{d\varphi} = \frac{df}{ds}\,\frac{ds}{d\varphi}
   =\frac{1+w'}{1-w'}.
$$

\begin{figure}[htb]
    \centering
   \includegraphics[width=5in]{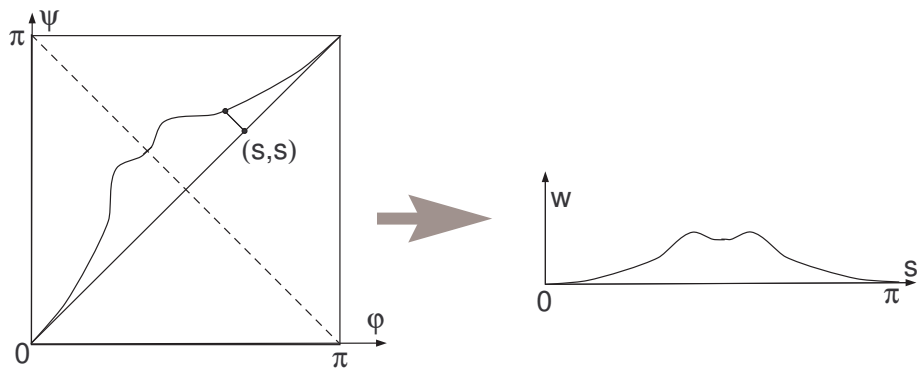}
    \caption{The graphs of $\psi = f(\varphi)$ and $w=w(s)$.}
    \label{FigGraph1}
\end{figure}

In these new variables the integral in Lemma~\ref{LmI2} becomes
\beq
    \mu_2 = \frac{1}{\pi}\,\int_0^\pi \log \frac{\sin (s+w)}{\sin (s-w)} (1-w') \, ds
\eeq

Note that
$$
    \int_0^\pi \log \frac{\sin (s+w)}{\sin (s-w)} \, ds = 0
$$
because of symmetry $w(\pi-s) = w(s)$.

We consider a function
\beq
    F(t):= \int_0^\pi \log \frac{\sin (s+tw)}{\sin (s-tw)} w' \, ds
\eeq
for $0 \leq t \leq 1$. Note that $F(0)=0$. We will show that $F'(t)>0$.
Indeed,
\begin{align*}
    2 F'(t) & = \int_0^\pi \left( \cot(s+wt) + \cot(s-wt) \right) \, d(w^2) \\
        & = \left.w^2\left( \cot(s+wt) + \cot(s-wt) \right) \right|_0^\pi \\
        & \;\;\;\; + \int_0^\pi \left( \frac{1+tw'}{\sin^2 (s+wt)} +
            \frac{1-tw'}{\sin^2 (s-wt)} \right) w^2 \, ds
\end{align*}
Recall that $w(0)=w(\pi)=0$ and $|w'|<1$. This makes the middle line
vanish. The integrand is always positive, which proves that $F'(t)>0$.
Then $F(t)>0$ for all $t>0$, and

$$ \mu_2 = -\tfrac{1}{\pi}\,F(1) < 0 \text{.}$$

Lemma~\ref{LmI2} is proved. \qed


\end{document}